\documentclass[a4paper,11pt]{amsart}
\usepackage{graphicx}
\usepackage{amsmath}
\usepackage{amssymb}
\usepackage{oldgerm}
\usepackage[all]{xy}
\newcommand{\supp}{\operatorname{supp}\nolimits}

\newcommand{\End}{\operatorname{End}\nolimits}
\newcommand{\ext}{\operatorname{Ext}\nolimits}
\newcommand{\Hom}{\operatorname{Hom}\nolimits}
\renewcommand{\Im}{\operatorname{Im}\nolimits}
\renewcommand{\mod}{\operatorname{mod}\nolimits}

\newcommand{\Ext}{\operatorname{Ext}\nolimits}

\newcommand{\Z}{\operatorname{\mathbb{Z}}\nolimits}
\newcommand{\Q}{\operatorname{\mathbb{Q}}\nolimits}

\newcommand{\N}{\operatorname{\mathbb{N}}\nolimits}

\newcommand{\gdim}{\operatorname{\underline{dim}}\nolimits}
\newcommand{\rep}{\operatorname{rep}\nolimits}

\newtheorem{theo}{Theorem}[section]

\newtheorem{cor}[theo]{Corollary}

\newtheorem{lemma}[theo]{Lemma}
\newtheorem{prop}[theo]{Proposition}
\newtheorem{defi}[theo]{Definition}
\newtheorem{rem}[theo]{Remark}
\newtheorem{example}[theo]{Example}

\newcommand{\I}{\mathcal{I}}
\newcommand{\T}{\mathcal{T}}

\newcommand{\cc}{\mathcal{C}_Q}

\begin{document}
\baselineskip=15pt
\title[Component cluster]{Component cluster for acyclic quiver}
\author{Sarah Scherotzke}
\address{S.~S.~: University of Bonn, Mathematisches Institut, Endenicher Allee 60, 53115 Bonn, Germany}

\email{sarah@math.uni-bonn.de}

\subjclass[2000]{}

\begin{abstract} 
The theory of Caldero-Chapoton algebras of Cerulli-Irelli, Labardini-Fragoso and Schr\"oer \cite{CerulliLabardiniSchroer}  
   leads to a refinement of the notions of cluster variables and clusters, via so called component clusters. 
 In this paper we compare component clusters to classical clusters for the cluster algebra of an acyclic quiver. 
We propose a definition of mutation between component clusters and determine the mutation relations of  component clusters for affine quivers.  
In the case of a wild quiver, we provide bounds for the size of component clusters.  
\end{abstract}

\maketitle

\section{introduction}
In \cite{CerulliLabardiniSchroer} Cerulli-Irelli, Labardini-Fragoso and Schr\"oer propose a broad generalization of the theory of cluster algebras   \cite{FominZelevinsky02}. 
They give a recipe to attach to any basic algebra $\Lambda$  a subalgebra 
$\mathcal{A}_{\Lambda}$ 
of a ring of rational functions:  
$\mathcal{A}_\Lambda$  is the \emph{Caldero-Chapoton algebra} of $\Lambda$. 
Similarly to cluster algebras, 
Caldero-Chapoton algebras come with 
an interesting collection of sets of generators which are called \emph{CC--clusters}. 
 In this paper, we investigate various  properties of Caldero-Chapoton algebras and CC--clusters in the special case when $\Lambda$ is the path algebra of an acyclic quiver.  

Note that if $Q$ is an acyclic quiver and $\Lambda = kQ$ is its path algebra, the Caldero-Chapoton algebra of $\Lambda$ is equal to the ordinary cluster algebra of $Q$, $\mathcal{A}_Q$. 
However even in this case the set of generators of $\mathcal{A}_Q$ that we obtain by viewing it as a Caldero-Chapoton algebra is larger than the set of classical cluster variables. Further, contrary to classical clusters, $CC$--clusters can have smaller cardinality than the vertex set of $Q$: the classical clusters of $\mathcal{A}_Q$ coincide with the CC-clusters of maximal size.


The construction of the $CC$--clusters in \cite{CerulliLabardiniSchroer} builds on work of Caldero, Chapoton and Keller  on the cluster character \cite{CalderoChapoton06}, \cite{CalderoKeller08}. 
The authors introduce  first  \emph{component clusters} which are families of irreducible components of the representation varieties of $\Lambda$ having some special properties. $CC$--clusters are then obtained by applying the Caldero-Chapoton map to the component clusters, see \cite{CerulliLabardiniSchroer} for a fuller explanation. 

In this paper we study the structure of component clusters when $\Lambda = kQ$, and $Q$ is an acyclic quiver.
We show that, as a consequence of Kac's Theorem \cite{Kac}, component clusters are in bijection with sets of pairwise $ext$-orthogonal distinct Schur roots.
Hence component clusters are closely linked to generic decompositions of dimension vectors \cite{Kac}, which have been studied also by Schofield \cite{Schofield} and more recently by Derksen and Weyman \cite{DerksenWeyman}.

When $Q$ is affine, we give a complete description of component clusters: they are either of size $n$ or $n-1$, where $n$ is the number of vertices of $Q$. Component clusters are of size $n-1$ if and only if they contain the unique positive isotropic Schur root. 
The situation is considerably more complicated when $Q$ is of wild type.
However, in this case, we are able to obtain an optimal upper bound for the number of imaginary Schur roots appearing in a component cluster. 
We also show that, if $Q$ is of wild type, we always have an infinite number of component clusters of size one. 
Further, motivated by the exchange relations between cluster variables, we give a definition of exchange relations between component clusters. 
For affine quivers, we explicitly determine these exchange relations.

The paper is structured as follows:  
In the first section we recall Kac's generic decomposition Theorem and classical results on root systems of quivers. We introduce negative Schur roots in order to define generic decompositions for any vector in $\Z^n$. 

In the second section, we determine the cluster components for affine quivers. In section three, we study the sizes of component clusters  
if $Q$ is of wild type. 
Finally, in the last section, we define mutations of component clusters and give an interpretation of exchange relations between two component clusters that are connected by a mutation. We work out the exact exchange relations for affine quivers.

{\bf Acknowledgements: } The author is grateful to Giovanni Cerulli-Irelli, Daniel Labardini-Fragoso and Jan Schr{\"o}er for useful conversations, and to Nicol{\' o} Sibilla for comments the first draft.

\section{Generalized generic decompositions and cluster components of quivers}
\subsection{Back to the roots}

In this section, we introduce notations  and recall some basic facts on root systems of acyclic quivers (see also  \cite{CrawleyBoevey92} and \cite{Schofield}).
Throughout this paper, $Q$ is a finite quiver without oriented cycles and $k$ a field. We denote the set of vertices by $Q_0$. We assume that the vertices are equipped with a total order and we  denote them 
$1, \ldots, n$. We denote the set of arrows $Q_1$. Furthermore we denote $s$ and $t$ the maps $s, t: Q_1 \to Q_0$ which send an arrow to its source and to its target respectively. 

To a dimension vector $d: Q_0 \to  \N$, we associate the variety of representations $$\rep_d Q:=\bigoplus_{a \in Q_1} \Hom(k^{d(s(a))}, k^{d(t(a))}).$$
It is a finite-dimensional vector space, hence an irreducible affine variety. There is a canonical action of $\prod_{i\in Q_0} Gl_{d(i)}(k)$ on $\rep_d Q$ having the property that the $G_d$ orbits 
are in bijection with the isomorphism classes of $kQ$-modules of dimension vector $d$. 

The \emph{support} of a dimension vector $d\in \N^n$ is the subset $$\supp(d):=\{i\in Q_0| d(i) \not=0\}$$ of $Q_0$. We say that the support is \emph{connected} if the full subquiver  of $Q$ generated by the vertices belonging to $\supp(d)$ is connected. 
A dimension vector is called \emph{a root}, if $\rep_d Q$ contains an indecomposable representation.  
A \emph{Schur root} is a dimension vector of a representation whose endomorphism ring is isomorphic to $k$. Such a representation is necessarily indecomposable and is called a  \emph{Schurian representation}. 

Let $d$ and $b$ be two dimension vectors. The functions 
$$hom(-,-): \rep_d Q \times \rep_b Q \to \N, \ \ (M, N)\mapsto \dim \hom_Q(M, N),$$ $$ext(-,-): \rep_d Q \times \rep_b Q \to \N \ \ (M, N)\mapsto \dim \ext^1_Q(M, N)$$ and $$end(-): \rep_d Q \to \N,\ \ N \mapsto \dim \End_Q(N)$$ are upper semi-continous.
Hence there are open subsets in $\rep_d Q \times \rep_b Q$ on which $ext$ and $hom$ are constant of minimal value and there is an open subset of $\rep_d Q$ on which $end$ is constant of minimal value. We set $ext(d,b)$, $hom(d,b)$ and $end(d)$ to be the minimal value of these functions. Note that all open 
subsets of an irreducible variety are  dense.

 We call two dimension vectors $a$ and $b$ \emph{ext-orthogonal} if $ext(a,b)$ and $ext(b,a)$ vanish. It also follows from upper semi-continuity that for any Schur root $d$ there is a dense open subset of Schurian representations in $\rep_d Q$.
Let $n$ be the cardinality of $Q_0$. 
The \emph{Euler form} is the bilinear form $$\langle -,-\rangle:\Z^n\times \Z^n \to \Z,\ (a,b) \mapsto \sum_{i \in Q_0} a_ib_i- \sum_{f\in Q_1} a_{s(f)}b_{t(f)}.$$ 
By \cite{Schofield}, the Euler form 
can alternatively be described by the formula $$\langle  a,b\rangle=\dim \Hom(M,N) -\dim \ext(N,M)$$ for any $M \in \rep_a Q$ and any $N \in \rep_b Q$. As two open sets intersect non-trivially in an irreducible variety, we also have the identity 
$$
\langle a,b\rangle=hom(a,b)-ext(a,b).
$$

The \emph{symmetrized Euler form} is the bilinear form on $\Z^n\times \Z^n$ given by $$(a,b)\mapsto \langle         a,b\rangle+\langle          b,a\rangle$$ and the \emph{Tits form} is the quadratic form $q(a):= \langle         a,a\rangle$ for all $a$ and $b\in \Z^n$. 
Roots are classified by their Tits form into three types. 
We refer to a root $d$ as \emph{real} if $q(d)=1$, as \emph{imaginary} if $q(d)\le 0$ and as \emph{isotropic} if $q(d)=0$. 
A representation $M$ in $\rep_d Q$ is \emph{rigid} if $\ext^1(M,M)$ vanishes. This is the case if and only if the representations isomorphic to $M$ form an open subset of $\rep_{d} Q$. If $d$ is a root and $\rep_{d} Q$ contains a rigid representation, then $d$ is a real Schur root. Conversely, if $d$ is a real Schur root then $\rep_{d} Q$ contains a  rigid representation $M$, which is necessarily Schurian. Further all Schurian representations of dimension vector $d$ are isomorphic to $M$.

Kac's generic decomposition theorem shows that Schur roots play an important role in understanding the variety of representations of a quiver:

\begin{theo}\cite{Kac}\begin{itemize}
\item[1.] Every dimension vector $d$ has a unique decomposition $$d= d_1 \oplus d_2 \oplus \cdots  \oplus d_s$$ as a sum of Schur roots $d_i$, such that the image of the natural embedding $$\coprod_{i=1}^s \rep_{d_i} Q \to \rep_d Q ,\  ( M_1, \ldots, M_s) \mapsto \bigoplus_{i=1}^s M_i$$ is an open set. In this case the generic extensions $ext(d_i,d_j)$ vanish for all $i\not =j$.

\item[2.] Conversely, every decomposition of $d$ into a sum of Schur roots $d_i$ such that the generic extensions $ext(d_i,d_j)$ vanish for all $i\not =j$, gives rise to an open embedding of $\coprod_{i=1}^s \rep_{d_i} Q $ into $\rep_d Q $.
\end{itemize}
\end{theo}
This unique decomposition of a dimension vector into a sum of Schur roots is called the \emph{generic decomposition}.

We will use the following standard notation. We denote by $P_i$, $I_i$ and $S_i$ respectively the projective indecomposable, injective indecomposable and simple module associated to the vertex $i \in Q_0$. 

\subsection{The cluster category}

Here we briefly summarize the relations between quiver representations and the theory of cluster algebras. We refer to \cite{Keller08c} for a fuller account. We assume that the ground field $k$ has characteristic $0$.
Let $\mathcal{D}_Q$ denote the bounded derived category of
$kQ$-modules. It is a triangulated category and
we denote its suspension functor by $\Sigma:\mathcal{D}_Q\to
\mathcal{D}_Q$. As $kQ$ has finite global dimension,
Auslander-Reiten triangles exist in $\mathcal{D}_Q$ by
Theorem 1.4 of \cite{Happel91a}. We denote
the Auslander-Reiten translation of $\mathcal{D}_Q$ by $\tau$. 
On non-projective modules, it coincides with the Auslander-Reiten
translation of $\mod kQ$. The
{\em cluster category} \cite{BuanMarshReinekeReitenTodorov06}
\[ \mathcal{C}_Q=\mathcal{D}_Q/(\tau^{-1} \Sigma)^{\Z}\] is the orbit
category of $\mathcal{D}_Q$ under the action of the cyclic group
generated by $\tau^{-1} \Sigma$. One can show \cite{Keller05} that
$\mathcal{C}_Q$ admits a canonical structure of triangulated
category such that the projection functor $\pi:\mathcal{D}_Q \to
\mathcal{C}_Q$ becomes a functor of triangulated categories. 

 We
refer to \cite{CalderoKeller08} for the definition of the
cluster character $L \mapsto X_L$ from the set of
isomorphism classes of objects $L$ of $\mathcal{C}_Q$ to the ring of Laurent polynomials 
$k[x_1^{\pm 1}, \ldots, x_n^{\pm 1}]$. We have $X_{\tau P_i}= x_i$ for all vertices $i$ of
$Q$ and $X_{M\oplus N}= X_M X_N$ for all objects $M$ and $N$ of
$\mathcal{C}_Q$.
We call an object $M$ in $\mathcal{C}_Q$ {\em rigid} if it has no
self-extensions, that is if the space $\ext^1_{\mathcal{C}_Q}(M,M)$ vanishes. The next Theorem explains in which way the cluster character 
allows us to view the cluster category as a categorification of $\mathcal{A}_Q$.

\begin{theo}[\cite{CalderoKeller08}] \label{acyclic-main-thm}\begin{itemize}
\item[a)] The map $L \mapsto X_L$ induces a bijection from the set
of isomorphism classes of rigid indecomposable objects of the cluster
category $\mathcal{C}_Q$ onto the set of cluster variables of the
cluster algebra $\mathcal{A}_Q$. \item[b)] If $L$ and $M$ are
indecomposables and $\Ext^1_{\mathcal{C}_Q}(L,M)$ is
one-dimensional, then we have  a generalized exchange relation
\begin{equation}\label{eq:acyclic-main-thm}
X_L X_M = X_E + X_{E'}
\end{equation}
where $E$ and $E'$ are the middle terms of the `unique' non split triangles
\[
\xymatrix{L \ar[r] & E \ar[r] & M \ar[r] & \Sigma L} \mbox{ and }
\xymatrix{M \ar[r] & E' \ar[r] & L \ar[r] & \Sigma M}
\]
\end{itemize}
\end{theo}

Let $L$ and $M$ be
two indecomposable objects in the cluster category such that
$\ext^1_{\mathcal{C}_Q} (M,L)$ is one dimensional. If both $L$ and
$M$ are rigid, then so are $E$ and $E'$,  and the sequence
(\ref{eq:acyclic-main-thm}) is an exchange relation of the cluster
algebra $\mathcal{A}_Q$. For this reason, in this case, we call the
triangles in (\ref{acyclic-main-thm}) {\em exchange triangles}. If $L$
or $M$ is not rigid, we call them {\em generalized exchange triangles}.

For all dimension vectors $d$ the cluster character is a constructible function on $\rep_d Q$. Hence it takes a constant value $X_d$ on an open subset of $\rep_d Q$. We call  $X_d$ the \emph{ generic cluster character} of $d$. The generic cluster characters have been conjectured to be a basis of the cluster algebra, called the generic basis (we refer to \cite{CerulliLabardiniSchroer} Section 1.2 for further details on this conjecture).  

\subsection{Generalized Schur roots and generic decompositions}\label{generalized decomposition} 

Let $$e_1, \ldots, e_n$$ be the standard basis of $\Z^n$. 
It will be useful to consider  also the \emph{negative Schur roots}, these are the elements $-e_i$, for $i\in\Q_0$.  
Indeed all indecomposable objects in the Cluster category $\mathcal{C}_Q$  are isomorphic to either a stalk complex of an indecomposable representation of $Q$ or to $\Sigma P_i$:  
we will interpret the negative Schur root $-e_i$ as the dimension vector of $\Sigma P_i$. 

An alternative point of view on negative Schur roots is given by  the decorated representations introduced in  \cite{DerksenWeymanZelevinsky}. 
Decorated representations yield a combinatorial construction  of the representations associated to negative Schur roots. In this article we rely instead on the categorical setup of cluster categories that was described in the previous paragraph.

Negative Schur roots allow us to define generic decompositions for any dimension vector with integer values. They are real Schur roots, as $\Sigma P_i$ has no self-extensions in the Cluster category. 
 We define a negative Schur root $-e_i$ to be ext-orthogonal to a positive Schur root $d$ if there is an $M \in \rep_d Q$ such that $\ext^1_{\mathcal{C}_Q}(M,\Sigma P_i)$ vanishes. This is the case if and only if $d_i$ vanishes. As $\ext^1_{\mathcal{C}_Q}(\Sigma P_i, \Sigma P_j)$ and $\ext^1_{\mathcal{C}_Q}(\Sigma P_j, \Sigma P_i)$ vanish for all $i\not =j$, all negative Schur roots are pairwise ext-orthogonal. 

We say that a general dimension vector $d: Q_0 \to \Z$ has a generic decomposition $d= d_1\oplus d_2 \oplus  \cdots \oplus d_s$ into generalized Schur roots $d_i$ if the $d_i$ are pairwise ext-orthogonal. A generic decomposition of $d$ always exists   and it is unique. If $d$ is non negative, then it coincides with Kac's generic decomposition. 

\subsection{Component cluster} Let $\Lambda$ be a basic algebra. 
Component clusters for $\Lambda$ have first been introduced in section 6 of \cite{CerulliLabardiniSchroer}. They are maximal collections of indecomposable strongly reduced components of the representation variety of $\Lambda$ with pairwise vanishing generic extensions. We refer to \cite{CerulliLabardiniSchroer} for additional details.

In the case where $\Lambda=kQ$ is the path algebra of $Q$, the representation varieties $\rep_v Q$, for a fixed dimension vector  $v$, are vector spaces and thus in particular   irreducible. Furthermore they are  strongly reduced. Also by Kac's generic decomposition Theorem $\rep_v Q$ is indecomposable if and only if $v$ is a Schur root. 

Hence, we can give an alternative definition of component clusters for 
$\Lambda = kQ$ in terms of roots. The \emph{component graph} of $Q$ is a graph with vertices corresponding to the Schur roots and arrows connecting two Schur roots $b$ and $d$ if and only if $b \not =d$ and they are ext-orthogonal.  The maximal complete subgraphs are called \emph{component clusters}. 

The component clusters consisting only of real Schur roots correspond to the classical clusters and are in bijection with the cluster-tilting objects of the cluster category. These objects have been studied extensively in the context of categorification of cluster algebras. In fact the cluster character establishes a bijection between the cluster-tilting objects and the clusters of $\mathcal{A}_Q$ (see Theorem \ref{acyclic-main-thm}).  We know by \cite{HappelUnger05a} that these component clusters have size $n$. Their mutations can be entirely  described using cluster combinatorics (see \cite{CalderoKeller08}).

In this paper we will consider all component clusters, and not just the component clusters corresponding to cluster-tilting objects: we will study their  structure, we will calculate their size, and we will provide an interpretation of their mutations.

\section{Component cluster for affine quivers}
We first determine the size and composition of component clusters of affine quivers. 
We refer to \cite{CrawleyBoevey92} for an introduction to the representation theory of affine quivers. 
The roots of affine quivers are either real or isotropic. Let $\delta$ denote the smallest  positive isotropic root. All other isotropic roots are $\Z$-multiples of $\delta$.

Considering the Auslander-Reiten component of $Q$ gives a natural classification of indecomposable representations in three types:
\begin{itemize}
\item The \emph{preprojective component} consists of $\tau^{-1}$-orbits of projective indecomposable modules. The roots associated to these representations satisfy $\langle -, \delta\rangle < 0$ and are real Schur roots. 

\item The \emph{ preinjective component} consists of $\tau$-orbits of injective indecomposable modules. The roots associated to these representations satisfy $\langle        -, \delta\rangle >0$ and are also real Schur roots. 

\item Finally the third type of representations are the \emph{regular} indecomposable modules appearing in tubes. They are $\tau$-periodic representations and the associated roots satisfy $\langle        -,\delta\rangle=0$.
\end{itemize}

We distinguish two types of tubes in the Auslander-Reiten quiver: the \emph{exceptional ones}, which are of size greater one and form a finite set,   and the \emph{homogenous tubes}, which are parametrized by the projective line.  
By \cite{CrawleyBoevey92} an indecomposable regular representation is Schurian if and only if its dimension vector is smaller or equal to $\delta$ (that is, all the entries of the dimension vector are smaller or equal to the entries of $\delta$).  

These are exactly the dimension vectors of the regular representations that lie in the first $p$ rows of a tube of rank $p$. The dimension vectors of the regular representations in the first $p-1$ rows 
are real Schur roots. The dimension vector of the regular representation in the row $p$ is always the isotropic root $\delta$. 
Hence there are infinitely many isomorphism classes of indecomposable modules with dimension vector $\delta$. It follows that $hom(\delta, \delta)$ vanishes and as a consequence the generic extension $ext(\delta, \delta)$ vanishes, even though every indecomposable module of dimension vector $\delta$ has non-vanishing self-extension.

The additive category of regular modules appearing in one tube is abelian and closed under extensions. Its simple objects are called \emph{regular simple} and the number of isomorphism classes of regular simple modules equals the rank of the tube. The indecomposable regular modules are uniserial with respect to the regular simple modules appearing in the same tube. The maximal rigid objects in the exceptional tubes have been described by Buan and Krause in \cite{BuanKrause} and \cite{BuanKrause1}. 
Corollary  3.8 of \cite{BuanKrause} and Corollary 2.4 and Theorem 5.2 of \cite{BuanKrause1} imply the next result. 
\begin{theo}\label{max-rigid-tube}
A maximal basic rigid object in a tube of rank $p$ has $p-1$ pairwise non-isomorphic indecomposable direct summands, each of which has at most $p-1$ regular simples in its regular composition series.
\end{theo} 
We determine next which Schur roots appearing in tubes are ext-orthogonal. We say that a Schur root belongs to a tube, if it is the dimension vector of a regular representation. 
\begin{lemma}\label{gen-ext-tube}
Two Schur roots belonging to different tubes are ext-orthogonal. 
The isotropic Schur root $\delta$ is ext-orthogonal to a Schur root $\alpha$ if and only if $\alpha$ is regular. 
\end{lemma}
\begin{proof}It is well-known that two indecomposable regular representations $A$ and $B$ lying in different tubes have no extension. 
As there exists a Schurian representation of dimension vector $\delta$ that does not appear in an exceptional tube, $\delta$ is ext-orthogonal to all regular roots. 

Let $d$ be a preinjective or preprojective root. In the preinjective case $\langle d, \delta\rangle $ is negative and in the preprojective case $\langle  \delta,d \rangle$ is negative. It follows that either $ext(d, \delta)$ or $ext(\delta, d)$ is non-zero. 
\end{proof}
We can now determine the component clusters. 
\begin{theo}\label{affine}
The component clusters are either of size $n$ or of size $n-1$. They are of size $n-1$ if and only if they contain $\delta$. 
\end{theo}
\begin{proof}
If $\delta$ is not contained in a component cluster, then the component cluster corresponds to a cluster-tilting object, hence it is of size $n$.
Suppose now that $\delta$ is contained in the component cluster. Then all other Schur roots in the component cluster belong to tubes and are real. In a tube of rank $p>1$ the maximal number of pairwise ext-orthogonal real Schur roots is $p-1$ by \ref{max-rigid-tube}. By Lemma \ref{gen-ext-tube} all Schur roots appearing in different tubes are ext-orthogonal. So a component cluster containing $\delta$ will also contain $p-1$ Schur roots for each exceptional tube of the Auslander-Reiten quiver. As the sum over all ranks minus 1 is equal to $n-2$ by \cite{CrawleyBoevey92}, the component clusters containing $\delta$ are of size $n-1$.
\end{proof}

Note that, as there are only finitely many regular Schur roots, there are only finitely many component clusters of size $n-1$ but infinitely many component clusters of size $n$. 

\begin{lemma} The $\Z$--span of  Schur roots appearing in a component cluster of size $n-1$ form a pure sub lattice of $\Z^n$  of rank $n-1$.  
\end{lemma}
\begin{proof}
Given $n$ pairwise ext-orthogonal real Schur roots, then their $\Z$--span is the entire lattice $\Z^n$. If $\delta, \alpha_1, \cdots, \alpha_{n-2}$ is a component cluster, then there is a 
dimension vector of a representation $\tau^{-l} P_e$ which is ext-orthogonal to $\alpha_1,\cdots,  \alpha_{n-2}$. 
This is equivalent to the fact that $\tau^{l+1}(\alpha_1+ \cdots +\alpha_{n-2})$ does not have support in $e$. Hence the $\Z$--span of $\tau^{l+1} \alpha_1, \cdots, \tau^{l+1} \alpha_{n-2}$ is the lattice ${0} \times U$, where $U$ is a pure sub lattice of rank $n-2$. It follows that $\tau^{l+1} \delta= \delta, \tau^{l+1} \alpha_1, \cdots, \tau^{l+1} \alpha_{n-2}$
span a pure sub lattice of rank $n-1$. As $\tau$ is a bijective integral linear form on $\Z^n$, the $\Z$-span of $\delta, \alpha_1, \cdots, \alpha_{n-2}$  is  a pure sub lattice of rank $n-1$.

\end{proof}

\section{Component cluster for wild quiver}
In this section we obtain an optimal bound for the maximal number of imaginary Schur roots appearing in a component cluster. 

The \emph{fundamental domain} $$\mathcal{F}:=\{ d\in \Z^m| (d, e_i)\le 0 \mbox{ for all }i\in\{1,\cdots,m\} \mbox{ and } \supp(d) \mbox{ is connected} \}$$ is a subset of the positive imaginary roots. We call these roots \emph{fundamental}. The set of positive imaginary roots is given by the image of the Weyl group action on $\mathcal{F}$. 
Note that the symmetrized Euler form is invariant under the Weyl group $W$: that is
 $$( \alpha, \beta ) = ( w \alpha,  w\beta )$$ for all $w \in W$. The set of positive imaginary roots is invariant under the action of $W$, but the set of real roots is not. Indeed, if $\alpha$ is a real Schur root, $w \alpha$ will not be positive in general. Furthermore, the Weyl group action does not map Schur roots to Schur roots and does not preserve ext-orthogonality. 

\begin{lemma}\label{fun}
Let $\alpha$ be a fundamental root. Then either 
$\alpha$  is isotropic and $\alpha= \bigoplus_{i=1}^n \beta$, where $n \in \N$ and $\beta$ is an isotropic fundamental Schur root, or $\alpha$ is a  Schur root.  
\end{lemma} 
\begin{proof}
If $\alpha$ is isotropic and fundamental, then its support is an affine quiver. As every affine quiver has a unique positive isotropic non-divisible positive root $\beta$, we have $\alpha=\bigoplus_{i=1}^n \beta$ for some $n \in \N$. 
Suppose that $\alpha$ is not a Schur root and is not isotropic, then by  Theorem 6.2 of \cite{Schofield}  it contains at least one real Schur root $\beta$ in its decomposition, and $(\alpha, \beta)$ is positive. But this is a contradiction to the fact that $\alpha$ is fundamental.  
\end{proof}

For any dimension vector $\alpha$, its \emph{null-cone} is give by 
$$N_{\alpha}:= \{ i \in Q_0| (e_i, \alpha)=0 \}.$$ 

We say that a dimension vector $\alpha$ is \emph{sincere} if all its entries are positive integers. 
\begin{lemma} Suppose that $\alpha$ lies in the fundamental domain and is sincere. Then either $Q$ is an affine quiver and $\alpha$ is isotropic, or the full sub-quiver on the set of vertices $N_{\alpha}$ is a union of Dynkin quivers. 
\end{lemma}
We have determined the component clusters of affine quivers in the previous section.

\begin{lemma}\label{hom}
Let us assume that either 
\begin{itemize}
\item  $\alpha$ and $\beta$ are two positive imaginary ext-orthogonal roots or 
\item  $\alpha$ is an imaginary fundamental root and $\beta$ is real such that $\alpha$ and $\beta$ are ext-orthogonal.  
\end{itemize}
Then $hom(\alpha, \beta)$, $hom(\beta, \alpha)$ and $(\alpha, \beta)$ vanish. 
\end{lemma}
\begin{proof}
In the first case, we can consider a Weyl group element $w$ 
such that $w\alpha$ lies in the fundamental domain. Then $w\beta$ is a positive root and we have that $(\alpha, \beta)=(w \alpha, w \beta)\le 0$. 
In the second case, since $\alpha$ is fundamental, $(\alpha, \beta) \le 0$. As $\alpha $ and $\beta$ are ext-orthogonal, we also have
$0=(\alpha, \beta)= hom(\alpha, \beta) + hom(\beta, \alpha)$. 
\end{proof}
Note that the previous lemma does not hold if, in the second part, we replace the assumption that $\alpha$ is fundamental with the assumption that $\alpha$ is imaginary. 
\begin{lemma}\label{sup}
Let $\alpha$ and $\beta$ be two positive imaginary roots which are ext-orthogonal. Suppose that $\alpha$ lies in the fundamental domain. Then the support of $\beta$ is totally disconnected from the support of $\alpha$. 
\end{lemma}

\begin{proof}
If $\beta$ and $\alpha$ are ext-orthogonal, then $(\alpha, \beta)=0$. Therefore the support of $\beta$ is totally disconnected from the support of $\alpha$ or it is contained in $N_{\alpha}\cap \supp \alpha$. Note that the quiver generated by the vertices of $N_{\alpha}\cap \supp \alpha$ is a Dynkin quiver.  Thus, since $\beta$ is an imaginary root, its support cannot be contained in $N_{\alpha} \cap \supp \alpha$ and this concludes the proof. 
\end{proof}
Let $\alpha$ be an imaginary Schur root which is fundamental and not sincere.  Then a component cluster contains $\alpha$ if and only if it contains all the negative Schur roots corresponding to the vertices connected to the support of $\alpha$.

\begin{lemma}\label{imaginary}
Let $\alpha_1, \ldots, \alpha_n$ be imaginary Schur roots appearing in the same component cluster. Then there exists a Weyl group element $w$ such that the $w\alpha_i$ are all fundamental and the support of $w\alpha_i$ and $w\alpha_j$ is totally disconnected for all $i \not =j$. Also,  there is a component cluster containing $w \alpha_1 , \ldots , w \alpha_n$. 
\end{lemma}
\begin{proof}
There is a Weyl group element $w_1$ such that $w_1\alpha_1$ is fundamental. Then all $w_1 \alpha_i$ are positive imaginary roots satisfying $(w_1 \alpha_i, w_1 \alpha_1)=0$. Hence the support of $w_1\alpha_i$ is totally disconnected from the support of $w_1\alpha_1$ for all $i \not =1$.

If we restrict $w_1\alpha_2$ to the quiver $Q_2$ generated by its support, then $w_1 \alpha_2$ is a positive imaginary root for that quiver. Hence there is  a Weyl group element $w_2$ which is a product of simple reflections on vertices of $Q_2$ such that $w_2 w_1\alpha_2$ is fundamental in $Q_2$. Then $w_1w_2 \alpha_2$ is also fundamental in $Q$ with support contained in $Q_2$ and $w_2w_1 \alpha_1=w_1\alpha_1$. It follows that the support of $w_2w_1\alpha_i$ is totally disconneted from the support of $w_2w_1\alpha_1$ and $w_2w_1 \alpha_2$ for all $i\not= 1,\ 2$. By induction on $n$, there is an element $w:= w_n\cdots w_1$ such that $w \alpha_i$ are all fundamental roots with pairwise totally disconnected support.
 
Roots with totally disconnected support are always ext-orthogonal and fundamental roots are always Schur by \ref{fun}. Hence there is a component cluster containing $w\alpha_1,\ldots , w\alpha_n$. 
\end{proof}
\begin{cor}\label{maximal imaginary}
The maximal number of imaginary Schur roots that can appear in a component cluster is given by the maximal number of totally disconnected subgraphs of wild or tame type. 
\end{cor}
\begin{proof}
By Theorem \ref{imaginary}, we can assume without loss of generality that the imaginary Schur roots in a cluster are fundamental and have totally disjoint support. The support of the roots are quivers of tame or wild type. 
\end{proof}
Note that by Corollary 21 of  \cite{DerksenWeyman} the number of real Schur roots in a component cluster is bounded by the number of vertices of $Q$ minus twice the number of imaginary Schur roots appearing in 
the component cluster.

\begin{lemma}
Let $\alpha_1, \ldots , \alpha_k, \beta_1, \ldots \beta_s$ be a component cluster such that $\alpha_1, \ldots , \alpha_k$ are imaginary non-isotropic Schur roots. Then for all $n \in \N$ $n\alpha_1, \ldots , n\alpha_k, \beta_1, \ldots \beta_s$ is also a component cluster.  
\end{lemma}
\begin{proof}
By Theorem 3.7 of \cite{Schofield} the $\N$-multiple of an imaginary non-isotropic Schur root $\alpha_i$ is also a Schur root. 
Let $\alpha$ and $\beta$ be two positive roots, and let $n \in \N$. We show that they are ext-orthogonal if and only if $n \alpha$ and $\beta$ are ext-ortogonal.  
If $\alpha$ and $\beta$ are ext-orthogonal there are representations $A$ and $B$ of dimension vector $\alpha$ and $\beta$ such that $\Ext^1(A, B)$ and $\Ext^1(B,A)$ vanish. Note that 
this holds if and only if 
$\Ext^1(\bigoplus_{i=1}^nA, B)$ and $\Ext^1(B,\bigoplus_{i=1}^nA)$ vanish. Thus $\alpha$ and $\beta$ are ext-orthogonal if and only if $n\alpha$ and $\beta$ are ext-orthogonal, and this concludes the proof. 
\end{proof}

\begin{rem}\label{wild}
For any wild quiver there is a sincere fundamental imaginary root $\alpha$ such that $N_{\alpha}$ is empty. 
Clearly, this Schur root appears as the only element of a component cluster.  As the null-cone of a root and the null-cone of its positive multiples agree, we conclude that wild quivers always have infinitely many component clusters of size one. 
\end{rem}

The next example shows that the size of component clusters depends also on the orientation of the quiver: suppose $\alpha$ is a Schur roots for two quivers  $Q$ and $Q'$ 
with isomorphic underlying (non-oriented) graph. Then the maximal size of component clusters containing $\alpha$ may be different for $Q$ and 
$Q'$.

\begin{example}
Let $\alpha$ be given by 
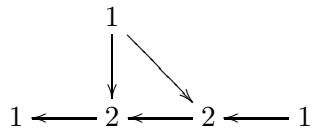
\[ \xymatrix{&1\ar[d]  \ar[rd]&&\\1& 2\ar[l]& 2\ar[l] & 1 \ar[l] .}\]
It is a fundamental root, and hence Schur. 

We change the orientation of one arrow and consider the fundamental Schur root $\beta$ 

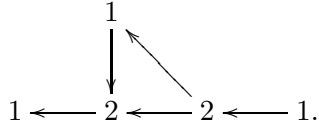
\[ \xymatrix{&1\ar[d] &&\\1& 2\ar[l]& 2\ar[l] \ar[lu]  & 1. \ar[l] }\] 

Direct computation shows that the component cluster containing $\alpha$ has exactly 2 elements, $\alpha$ and $\alpha'$ where $\alpha'$ is given by 
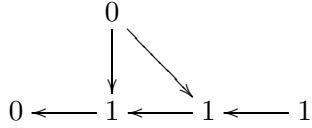
\[  \xymatrix{&0\ar[d]  \ar[rd]&&\\ 0& 1\ar[l]& 1\ar[l] & 1 \ar[l] .}\] 

On the other hand $\beta$ appears alone in a component cluster. 
Note also that by \ref{maximal imaginary} the maximal number of imaginary roots appearing in the same component cluster is at most one. 
\end{example} 

\begin{rem}
As the orientation of the quiver affects  the size of component clusters but the Tits form is independent of it, we cannot hope for an exact upper bound involving the Tits form of a root. Another way to see this is as follows. Start with a fundamental sincere root  $\alpha$ of a quiver $Q$ and add a vertex $x$ to $Q$ and $n>2$ arrows from $x$ to $y$, where $y$ is a vertex of $Q$ which is totally disconnected from $N_{\alpha}$. Let us denote the new quiver by $Q'$ and let $\alpha'$ be a new root with $\alpha'(z)=\alpha(z)$ for all $z \in Q_0$ and $\alpha'(x)=1$. Then the root $\alpha'$ is fundamental and sincere and there is a canonical bijection between the component clusters containing $\alpha$ and the component clusters containing $\alpha'$  but $q(\alpha') $ can be made arbitrarily small by increasing $n$. 
\end{rem}

\begin{lemma}
Let $\alpha$ be a fundamental non-divisible isotropic root. Then $\alpha$ appears in a component cluster of size $|Q_0|-1$.  
\end{lemma}
\begin{proof}
The support of $\alpha$ is an affine quiver and we know by Theorem \ref{affine} that $\alpha$ is ext-orthogonal to $|\supp(\alpha)|-1$ real Schur roots that 
have support contained in $\supp ( \alpha)$. Now for every vertex connected to $\supp( \alpha)$ we add the negative Schur root $-e_j$ to this collection. We can now complete the $ext$-orthogonal collection by real Schur roots with support in the vertices totally disconnected to $\supp(\alpha)$. 
\end{proof}
Note that the previous Lemma is false in general if we drop the hypothesis that $\alpha$ is fundamental. Also, non-divisible isotropic roots are not necessarily Schur ( see Example 27 of  \cite{DerksenWeyman}).

It is clear by the uniqueness of the generic decomposition that $n$ Schur roots appearing in the same component cluster are linearly independent.

\section{Mutation of component clusters}
Motivated by the cluster mutations and exchange relations which appear in the definition of cluster algebras, we propose a definition of mutations and exchange relations of component clusters. 
\begin{defi}
Two component clusters $ C_1$ and $C_2$ are connected by a \emph{ mutation} if their intersection has cardinality min$(|C_1|, |C_2|)-1$. 
\end{defi}
If the components  cluster consists of real Schur roots, then they corresponds uniquely to clusters of the cluster algebra $\mathcal{A_Q}$, and the above definition recovers the ordinary definition of cluster mutation.  
\begin{prop}
The mutation graph is connected.
\end{prop}
\begin{proof}
Clearly by the Bongartz completion, there is a path of length at most $n$ from a component cluster to a classical cluster. By \cite{HappelUnger05a} the full subgraph consisting of classical clusters 
is connected by mutation. Hence all component clusters are connected by mutation. 
\end{proof}
In order to define exchange relations, we recall a few preliminary results. 
\begin{lemma}\label{ext}
Let $N$ and $M $ be two $kQ$-modules. Then we have a canonical
isomorphism $\ext^1_{\mathcal{C}_Q} (M,N)\cong \ext_{kQ}^1(M, N)
\oplus D \ext_{kQ}^1(N,M)$.
\end{lemma}
\begin{proof} This is Proposition 1.7 c) of
\cite{BuanMarshReinekeReitenTodorov06}.
\end{proof}

Let $C_1$ and $C_2$ be two clusters connected by mutation. Then there exist unique roots $\alpha, \alpha'$ such that $\{\alpha \} = C_1-C_2$ and $\{ \alpha' \}=C_2- C_1$.
An \emph{exchange relation} between $C_1$ and $C_2$  is a 
polynomial equation in the cluster algebra $\mathcal{A}_{Q}$ relating 
the cluster characters $X_{\alpha}$, $X_{\alpha'}$ and $X_d  $ for $d \in C_1 \cap C_2$. 
Here we are working one categorical level up, on the level of roots. Hence for us 
an exchange relation will be given by a generic decomposition of 
$\alpha + \alpha'$, where $\alpha \in C_2-C_1 $ and $\alpha' \in C_1-C_2$.  We will show that the generic decomposition involves roots which are ext-orthogonal to all 
roots in $C_1 \cap C_2$.

\begin{lemma}\label{ext-orthogonal}
Let $\alpha$ and $\beta$ be two roots which are ext-orthogonal to a root $d$ and suppose that $ext(\alpha, \beta)\not=0$. Then there are open subsets $U_{\alpha}$, $U_{\beta}$ and $U_d$ 
of $\rep_{\alpha} Q$, $\rep_{\beta} Q$ and $\rep_d Q$ respectively such that for all $A \in U_{\alpha}$ and $C\in  U_{\beta}$ and all non-split triangles 
\[A \to B \to C \to \Sigma A\] and \[C \to B' \to A \to \Sigma C,\] the spaces $\ext^1_{\mathcal{C}_Q}(B, D)$ and $\ext^1_{\mathcal{C}_Q}(B',D)$ vanish for all $D \in  U_d$.
\end{lemma}
\begin{proof}
By the irreducibility of the varieties of representations, there are open subsets $U_{\alpha}$, $U_d$ and $U_{\beta}$ of $\rep_{\alpha} Q$, $\rep_d Q$ and $\rep_{\beta} Q$ respectively such that $\ext_{\mathcal{C}_Q}^1(A,D)=\ext_{\mathcal{C}_Q}^1(C, D)=0$ by \ref{ext} and $\ext_{\mathcal{C}_Q}^1(A, C)\not=0$ for all $A\in U_a$, $C\in U_c$ and $D\in U_d$. So by the $2$-Calabi-Yau property we obtain the existence of two non-split triangles 
\[A \to B \to C \to  \Sigma C\] and \[C \to B' \to A \to \Sigma A\] in $\mathcal{C}_Q$. Applying the functor $\Hom(-, D) $ to the distinguished triangles allows us to conclude that $\ext^1_{\mathcal{C}_Q}(B, D)$ and $\ext^1_{\mathcal{C}_Q}(B',D)$ vanish for all $D \in  U_d$. 
\end{proof}

\begin{prop}Let $ \alpha_1, \ldots, \alpha_{n}$ be a collection of ext-orthogonal Schur roots. Suppose that $ \alpha_1'\not =\alpha_1$ is a Schur root ext-orthogonal to  $ \alpha_2, \ldots, \alpha_{n}$ and $ext(\alpha_1, \alpha_1')$ does not vanish. Then the generic decomposition of $\alpha_1+ \alpha_1'$ involves  only Schur roots that are different from both $\alpha_1$ or $\alpha_1'$ and are ext-orthogonal to $ \alpha_2, \ldots, \alpha_{n}$. 
\end{prop}
\begin{proof}As $ext(\alpha_1, \alpha_1') $ does not vanish, there are open sets $U_{\alpha_1}$ and $U_{\alpha_1'}$ such that, for all $A_1 \in U_{\alpha_1}$ and $A_1' \in  U_{\alpha_1'}$, the space $\ext^1(A_1, A_1')$ does not vanish. Hence there is a non-split exact sequences $$0 \to A_1 \to A \to A_1' \to 0$$ and $A$ has dimension vector $\alpha_1+ \alpha_1'$. 
By Lemma \ref{ext-orthogonal}, for all $ i=2, \ldots, n$ there is a representation $A_i$ with dimension vector $\alpha_i$  such that $\ext^1_{\mathcal{C}_Q}(A, A_i)$ vanishes. Let $d_1\oplus  \ldots \oplus d_s$ be a generic decomposition of $\alpha_1+\alpha_1'$. Then we have that the $d_i$-s and the $\alpha_j$-s are all pairwise ext-orthogonal. Further 
 the $d_i$-s are different from both $\alpha_1$ and $\alpha_1'$ by  Theorem 3.3 of  \cite{Schofield}. 
\end{proof}
The next statement follows now immediately from the previous two. 
\begin{theo}\label{exchange relation}
Let $C_1$ and $C_2$ be two component clusters that are related by mutation. Then for every pair $(\alpha, \alpha')$ with $\alpha \in C_1-C_2$ and $\alpha' \in C_2-C_1$ there is a component cluster $C_3$ containing $C_1 \cap C_2$ such that all Schur roots in the generic decomposition of $\alpha + \alpha'$ are contained in $C_3$. 
\end{theo}
Hence we obtain an \emph{exchange relation} between two component clusters which are related by mutation for every pair $\alpha \in C_1-C_2$ and $\alpha'\in C_2-C_1$ in terms of a third cluster $C_3$ containing $C_1 \cap C_2$ and all Schur roots appearing in the generic decomposition of $\alpha+\alpha'$. 
In the case of classical clusters containing only real Schur roots, we have a more precise result. The Schur roots in a decomposition of $\alpha+\alpha'$ are contained in the intersection $C_1 \cap C_2$. In the next Section we will see that this result cannot be extended to component clusters, as it fails in the case of affine quiver.

\subsection{Exchange relations for affine quivers}

In the case of an affine quiver, we have concrete descriptions of the component clusters. We will use these to obtain exchange relations between the 
cluster character of Schur roots appearing in component clusters which are related by mutation. 

In the first part, we work out the exchange relations arising from mutation between a component cluster of size $n$ and a component cluster of size $n-1$.  In the second part we will consider exchange relations arising from mutation between  two component clusters of size $n-1$. 

\begin{lemma}[Theorem 3.14 of \cite{Dupont}] \label{lemma:x-delta}
Let $N$ and $M$ be two regular simple $kQ$-modules whose dimension
vectors equal $\delta$. Then $X_M$ equals $X_N$.
\end{lemma}
The regular simple modules of dimension vector $\delta$ form an open subset of $\rep_{\delta} Q$. Hence the generic cluster character $X_{\delta}$ equals $X_M$ for any regular simple module $M$ with dimension vector $\delta$.

Recall that a vertex $e$ of $Q$ is \emph{extending} if $\delta_e=1$. 
\begin{lemma}\label{extend}
Let $\delta, \alpha_1, \cdots, \alpha_{n-2}$ be a component cluster. Then there exists a positive Schur root $\beta \not= \delta$ such that $\beta$ is ext-orthogonal to $\alpha_1, \cdots, \alpha_{n-2}$. In this case $\beta$ is either the dimension vector of the preprojective module  $ \tau^{-l} P_e$ or the dimension vector of the preinjective module $\tau^{l} I_e$, where $l \in \N$ and $e$ is an extending vertex. 
\end{lemma}
\begin{proof}The existence of $\beta$ is clear by \cite{HappelUnger05a}. As $\beta$ is a real root it is either preprojective or preinjective. 
So either $\beta$ is the dimension vector either of the preprojective module $ \tau^{-l} P_e$ or of the preinjective module $\tau^{l} I_e$ for some positive integer $l$ and some vertex $e$ of $Q$. 

Using the Auslander formula, the ext-vanishing condition is equivalent to the vanishing of $hom(\gdim \tau^{-l} P_e, \tau \alpha_i)$ or  $hom( \tau^{-1} \alpha_i,\gdim \tau^{l} I_e)$ respectively for all $1\le i \le n-2$. Both conditions are equivalent to the fact that $$\tau^{(l+1)}(\alpha_1+\cdots +\alpha_{n-2})$$ has no support in $e$. It remains to show that $e$ is an extending vertex. 
If $Q$ is an orientation of a Kronecker quiver or of $\tilde A_n$, there is nothing to show, as every vertex is extending.

In the remaining cases the Auslander-Reiten quiver contains at least one exceptional tube of size $2$. Let $\alpha$ and $\beta:=\tau(\alpha)$ denote the dimension vectors of the regular simples in such a tube. We assume without loss of generality that $\alpha$ belongs to $\alpha_1, \cdots , \alpha_{n-2}$. Then $\alpha$ has no support in $e$ by the first part of the proof. As $\alpha$ and $\beta$ are roots, their supports have to be connected. As $\alpha+\beta=\delta$ by \cite{CrawleyBoevey92} and $\delta$ is sincere, we know that $e $ is contained in the support of $\beta$. So the support of $\alpha$ and $\beta$ is disconnected and their supports are linked by one arrow $e' \to e$ with $\alpha(e' )$ non-zero. As $\alpha$ is a real Schur root, we have $$hom(\alpha, \beta)=ext(\alpha,\alpha)=0 \mbox{ and } ext(\alpha, \beta)=hom(\alpha,\alpha)=1.$$ Suppose $e$ is not extending. Then the following inequality $$2\le \delta(e)\delta(e')=-\langle       \alpha, \beta\rangle=ext(\alpha, \beta)- hom(\alpha, \beta)=1$$ gives a contradiction. Hence $e$ needs to be an extending vertex.  
\end{proof}
We denote $g$ the smallest common multiple of the tube lengths. 
\begin{cor} 
Let $\alpha_1, \cdots, \alpha_{n-2}$ be a collection of pairwise ext-orthogonal exceptional Schur roots and let $\beta$ be a preprojective Schur root which is ext-orthogonal to this collection. Then, 
for all $m \in \N$, $\tau^{-mg} \beta$ is also ext-orthogonal to $\alpha_1, \cdots, \alpha_{n-2}$. 
\end{cor}
\begin{proof}
Clearly, $\tau$ preserves ext-orthogonality and $\tau^g$ acts as the identity on regular modules. Hence,  
for all $m \in \N$,  $\tau^{-mg} \beta$ is also ext-orthogonal to $\alpha_1, \cdots, \alpha_{n-2}$. 
\end{proof}
From this result it follows immediately that there are infinitely many clusters which are connected by mutation to a component cluster containing $\delta$. 
The next theorem gives the exchange relations between a component cluster and a cluster. In this case, we also obtain exchange relations of generic cluster characters. 
\begin{theo}
Let $\delta, \alpha_1, \cdots , \alpha_{n-2}$ be a component cluster. Let $\beta$ be a preprojective Schur root such that $\beta, \alpha_1, \cdots , \alpha_{n-2}$ is a collection of pairwise ext-orthogonal Schur roots. Then there are exactly two completions $\beta, \beta_1, \alpha_1,\cdots , \alpha_{n-2}$ and $\beta, \beta_1', \alpha_1,\cdots , \alpha_{n-2}$ to clusters satisfying:
\begin{itemize} 
\item $ \delta + \beta= \beta_1$ and $\beta_1$ is the dimension vector of a  preprojective module;
\item  either $\beta_1' $ 
is the dimension vector of a preinjective module, or $\beta_1'$ is the dimension vector of a preprojective module and $ \delta+\beta_1' =\beta$;
 
\item  the generic cluster characters satisfy $X_{\delta} X_{\beta}= X_{\beta_1}+X_{\beta_1'}.$
\end{itemize}
\end{theo}
\begin{proof}By Lemma \ref{extend} the real Schur root $\beta$ is the dimension vector of a module in the $\tau$-orbit of the projective indecomposable module associated with an extending vertex $e$.  Therefore $\langle   \delta,    \beta \rangle=\delta_e=1$ and for every indecomposable regular simple representation $C\in \rep_{ \delta}Q $ and every indecomposable representation $A\in \rep_{\beta} Q$, there is a non-split exact sequence $$0 \to A\to B \to C \to 0 .$$
By \cite{CrawleyBoevey92} the module $B$ is preprojective and indecomposable and therefore its dimension vector is a Schur root which we denote by $\beta_1$. As $ext(\beta, \beta)=ext(\beta, \delta)=0$, we also have that $ext(\beta, \beta_1)=0$. 

From  $1=\langle \beta, \delta\rangle+ \langle \delta, \delta\rangle = \langle \beta_1, \delta\rangle$ and $1= \langle \beta_1, \beta_1\rangle = \langle \beta_1, \delta\rangle + \langle \beta_1, \beta\rangle$, we deduce that $\langle \beta_1, \beta\rangle$ vanishes.
We conclude from the vanishing of $ext(\beta, \beta_1)$ that every non-zero map in $hom(\beta_1, \beta) $ has to be surjective. As $\beta$ is a Schur root, $hom(\beta_1, \beta)$ vanishes and so does $ext(\beta_1, \beta)$. Therefore $\beta$ and $\beta_1$ are ext-orthogonal. We conclude by Lemma \ref{exchange relation} that $\beta_1, \beta, \alpha_1, \cdots , \alpha_{m-2}$ is a cluster. 

By the 2-Calabi-Yau property of the Cluster category, there exists a non-split triangle $$C \to B' \to A  \stackrel{f} \to \Sigma C.$$ In the Cluster category, the object $\Sigma C$ is isomorphic to $\tau(C)\cong C$. Hence we have $\Hom_{\cc}(A, C) = \bigoplus_{i \in \Z} \Hom_{D^b(Q)} (A, \Sigma^i C)$ and as $ \Ext^1(A, C)$ vanishes, we can view $f$ as a morphism of modules $f: A \to C$.
Then the object $B'$ splits into $B_1'\oplus \Sigma^{-1}B_2'$, where $B'_1$ is isomorphic to the kernel of $f$ and $B'_2$ is isomorphic to the cokernel of $f$.

We assume first that $B_1'$ does not vanish.
Clearly, $B_1'$ is preprojective and indecomposable as $\langle   \gdim    B_1', \delta\rangle=1$. The dimension vector of $B_1'$ is therefore a Schur root which we denote by $\beta_1'$. If $B_1'$ does not vanish the image of $f$ is a regular module and hence $B_2'$ is also a regular module of dimension vector smaller than $\delta$. As the  generic hom-space between $\delta$ and any exceptional Schur root vanishes, $B_2'$ vanishes and $f$ is surjective. 

We have that $\beta_1'$ is ext-orthogonal to $\beta$ as  can be seen by applying $\langle  \beta, -\rangle$ to the exact sequence $0 \to \ker f \to A \to \Im f \to 0$. 
Then $ext(\beta, \beta_1)=\langle       \beta, \beta_1\rangle=\langle       \beta, \beta\rangle-\langle       \beta, \delta\rangle=0$. Furthermore we have $ ext(\beta_1', \beta) \le ext(\beta_1', \beta_1') + ext(\beta_1',  \gdim \Im f)$ and the last term vanishes as $\Im f$ is regular. 
Hence $ \beta, \beta_1', \alpha_1, \cdots, \alpha_{n-2}$ is a cluster and $\beta_1' +\delta= \beta$.

If $\ker f$ vanishes, we have that the cokernel of $f$ satisfies $\langle     \gdim  B_2', \delta\rangle =-1$, hence it has a preinjective direct summand. Applying $\hom(-, B_2')$ induces the exact sequence $$0\to \Hom(B_2', B_2') \to \Hom(C, B_2') \to \Hom(A, B_2')=0.$$ As $hom(\delta, \gdim B_2')=\langle       \delta, \gdim B_2'\rangle=1$, the module $B'_2$ is indecomposable and its dimension vector is a real Schur root. If $B_2'$ is not injective, then the Schur root to $\tau^{-1} B_2'$, extends the ext-orthogonal collection $\alpha_1, \cdots, \alpha_{n-2}$. Furthermore, we have $hom(\beta, \gdim B_2')=\langle       \beta, \beta_2'\rangle=\langle       \beta, \delta\rangle-\langle       \beta, \beta\rangle=0$.
Hence $\tau^{-1} B_2'$ is ext-orthogonal to $\beta$ and its Schur root $\beta_1'$ completes $\beta, \alpha_1, \cdots, \alpha_{n-2}$ to a cluster.

If $B_2'$ is the injective module associated to the vertex $i$, then the roots $\beta, \alpha_1, \cdots, \alpha_{n-2}$ have vanishing support in $i$. Hence the Schur root $-e_i$ associated with the decorated representation $\Sigma P_i$ completes $\beta, \alpha_1, \cdots, \alpha_{n-2}$ to a cluster.

Finally, the multiplication formula yields the relation 
$$X_C X_B=X_{B_1}+X_{B_1'}.$$ 

As $B$, $B_1$ and $B_1'$ are indecomposable and rigid, their cluster characters equal the generic cluster character of their Schur roots. By Lemma \ref{lemma:x-delta}, we also have $X_{C} = X_{\delta}$. This finishes the proof.  
\end{proof}
\begin{rem}
Note that if $\beta_1'$ is preinjective we obtain similar exchange relations: there is a unique preinjective root $\beta_1''$ such that 
$\beta_1' , \beta_1'', \alpha_1, \cdots, \alpha_{n-2}$ is a cluster and $\delta+ \beta_1' =\beta_1''$.  The proof is similar. 
\end{rem}

Next we study the exchange relations between two component clusters of size $n-1$. It is useful to restrict first to ext-orthogonal collections of Schur roots appearing in the same exceptional tube $\T$. Let $\T$ be of rank $m$ and let $S \in \T$ be a regular simple module.

In order to study exchange relations between two regular clusters, we need to introduce the combinatorics of the appendix A of \cite{BuanKrause1}. We consider the intervals $[i,j]:=\{i, i+1, \ldots, j\}$ mod $m+1$ for $i, j \in \{0, \cdots, m\}$ with $i \not=j$.  Let us denote $\I(m)$ the set of all these intervals. We call two intervals \emph{compatible} if as sets they are either disjoint or one is  a subsets of the other. Then there is a bijection between the Schur roots of $\T$ and $\I(m)$ sending $[i,j]$ to the Schur root of the indecomposable representation with regular composition series $\tau^{-i} S, \ldots, \tau^{-j+1}S$. Then the Schur root $\delta$ corresponds to the interval $[0,m]$. 
The proof of the following fact is elementary.

\begin{lemma}
Two Schur roots in $\T$ are ext-orthogonal if and only if the corresponding intervals are compatible. Every set of compatible intervals can be completed to a set of $m$ compatible intervals. 
\end{lemma} 
We consider next the set $\mathcal{B}$ with elements the maximal sets of compatible intervals containing $[0,m]$.

\begin{lemma}\label{tube mutation}
Let $\alpha_1, \ldots, \alpha_{m}$ be a maximal set of ext-orthogonal Schur roots in $\T$.  Then there is exactly one Schur root $\alpha_1'\not = \alpha_1$ in $\T$ such that $\alpha_1', \alpha_2, \cdots, \alpha_{m}$ is a maximal ext-orthogonal collection.

Furthermore there are at most two distinct Schur roots  
$$\alpha,\ \beta \in \{\delta, \alpha_2, \cdots, \alpha_{m} \} $$ such that there is up to isomorphism exactly one non-split exact sequence 
$$0 \to A_1 \to A \to A_1' \to 0 $$ where $A_1$ and $A_1'$ are indecomposable regular representations with dimension vectors $\alpha_1$ and $\alpha_1'$ and $A$ is the direct sum of two indecomposable regular representations  of $\T$ with dimension vectors $\alpha$ and $\beta$. 
\end{lemma}
\begin{proof}
Let $I$ be the interval corresponding to $\alpha_1$ and let $Z \in \mathcal{B}$ be the set of maximal compatible intervals containing the $m$ intervals associated to $\alpha_1, \ldots, \alpha_m$.  Without loss of generality, we can assume that inf $I=0$. Then there is an interval $I^+$ in $Z-\{I\}$ such that either  inf $I =$ inf $I^+$ or sup $ I= $sup $I^+$. We assume without loss of generality that the first case holds and pick the smallest interval $I^+$, satisfying that property.
By the compatibility, we assume that $I$ is also a subset of $I^+$. The converse case of $I^+$ being contained in $I$ can be treated similarly. 
Set $I':=[i,j]$ where $j:=$sup $I^+$ and $i:= \text{min}\{ \text{ inf }S  | S \in Z-\{I\} , \text{sup }S= \text{sup }I \}$ or $i =$ sup $I$ if the the set is empty. 
Then $I'$ is the unique interval different from $I$ and compatible with $Z-\{I\}$. 
We can see this as follows. Assume that there is another interval $I''\not =I'$ compatible with $Z-\{I\}$. Then $I''$ is not compatible with $I$ and $I'$.  Furthermore by the compatibility with $Z-\{I\}$, we have that $I'' $ is contained in $I^{+}$. Hence we find 

 $$\mbox{ inf }I < \mbox{ inf }I'' < \mbox{inf }  I' \le \mbox{ sup } I < \mbox{sup } I '' \le \mbox{ sup }I' .$$ Then $A:=[\mbox{ inf }I'',\mbox{ sup }I']$ is compatible with $Z-\{I\} \cup\{I'\}$. But $A$ does not lie in $Z-\{I\}$ as it is not compatible with $I$. Hence we obtain a contradiction to the assumption that $Z$ is maximal. 

We have that $I^+=I\cup I'$ and $I^-:= I\cap I'$ is also contained in $Z$, as it is compatible by construction with all intervals in $Z-\{I\}$.  
If $I^-$ consists of only one point, then it is not an element of $\mathcal{I}$ and we ignore it. 

Let $\alpha_1'$, $\alpha$ and $\beta$ be the Schur roots corresponding to $I'$, $I^+$ and $I^-$. Furthermore let $A_1'$ and $A_1$ be the Schurian representations associated to $\alpha_1$ and $\alpha_1'$ and let $A$ be the direct sum of the two indecomposable representations associated to the roots $\alpha$ and $\beta$. Then there exists a non-split exact sequence $$0 \to A_1 \to A \to A_1' \to 0.$$ It is uniquely determined up to isomorphism as extensions between two Schurian representations in a tube are at most one-dimensional. 
The second case follows analogously.
\end{proof}
Note that it is not clear whether we can  obtain exchange relations on the level of generic cluster characters. Indeed $A$ could have an indecomposable direct summand $C$ of dimension vector $\delta$. Then $C$ would not be a regular simple representation, and in this case it is not known wether $X_C$ is equal to the generic cluster character $X_{\delta}$.
\begin{theo}
Let $\delta, \alpha_1,\ldots, \alpha_{n-2}$ be a component cluster ordered in such a way that $\alpha_1,\ldots, \alpha_m$ belong to the same tube. Then there is exactly one Schur root $\alpha_1'\not=\alpha_1$ such that 
$\delta, \alpha_1',\ldots, \alpha_{n-2}$ is a component cluster. 

In this case $\alpha_1$ and $\alpha_1'$ belong to the same tube and $\alpha_1+ \alpha_1'$ has a generic decomposition as a direct sum of either one or two Schur roots in $\{\delta, \alpha_2, \ldots, \alpha_m \}$. 
\end{theo}
\begin{proof}
By Lemma \ref{tube mutation}, there is exactly one Schur root $\alpha_1'$ different from $\alpha_1$ which belongs to the same tube of $\alpha_1$, and which completes $\delta, \alpha_2,\ldots, \alpha_m$ to a component cluster. The second part follows immediately from the previous Lemma. 
\end{proof}

\def\cprime{$'$} \def\cprime{$'$}
\providecommand{\bysame}{\leavevmode\hbox to3em{\hrulefill}\thinspace}
\providecommand{\MR}{\relax\ifhmode\unskip\space\fi MR }
\providecommand{\MRhref}[2]{%
  \href{http://www.ams.org/mathscinet-getitem?mr=#1}{#2}
}
\providecommand{\href}[2]{#2}

\end{document}